\documentclass[reqno,12pt]{amsart}
\usepackage{cases}
\usepackage{mathrsfs}
\usepackage[T1]{fontenc}
\usepackage{mathrsfs}
\usepackage{amsmath,latexsym,amssymb,amsfonts,amsbsy,amsthm}
\usepackage{bm}
\usepackage[usenames]{color}
\usepackage{xspace,colortbl}
\usepackage{epsfig}
\usepackage{graphicx}
\usepackage{float}
\usepackage{subfigure}
\usepackage{amsmath,amsfonts,amsthm,amssymb,amscd}
\usepackage{color}
\usepackage[title,titletoc,toc]{appendix}
\usepackage{bigints}
\allowdisplaybreaks[4]

\newcommand\si{\sigma}

\newcommand\na{\nabla}

\newcommand\va{\varepsilon}

\newcommand\Om{\Omega}

\newcommand\gl{\geqslant}
\newcommand\ls{\leqslant}

\newcommand\om{\omega}

\newcommand\ri{\rightarrow}

\newcommand\lge{\langle}
\newcommand\rge{\rangle}

\newtheorem{remark}{Remark}[section]
\newtheorem{lemma}{Lemma}[section]
\newtheorem{definition}{Definition}[section]

\newtheorem{theorem}{Theorem}[section]

\title[Sign-changing solutions of variational inequality]
{Sign-changing solutions of variational inequality}

\author[Xu Xian, Xian Xu]{}
\thanks{This work is supported by the NSFC 11871250.}
\thanks{$*$ the corresponding author. }

\begin{document}
 \maketitle

\centerline{\scshape Xian Xu$^1$,  Taotao Wang$^1$, Baoxia Qin$^{2,*}$ }
\medskip
{\footnotesize
 \centerline{$^1$School of Mathematics and Statistics, Jiangsu Normal University }
  \centerline{Xuzhou, Jiangsu 221116, China}

  \centerline{$^2$School of Mathematics,
Qilu Normal  University }
  \centerline{Jinan,
Shandong,
250013,  P. R. China} }


\medskip
\begin{abstract}
Variational inequality problems have extensive and important applications.  In this paper, we used penalty method and  the method of invariant set of descending flow to obtain the existence results for solutions of a variational inequality. In particular, we obtain the existence results for sign-changing solutions of variational inequalities for the first time based on the method of invariant set of descending flow.


\end{abstract}


\renewcommand{\theequation}{\thesection.\arabic{equation}}
\setcounter{equation}{0}

\section{Introduction}
\rm In this paper, we will study the following variational inequality
\begin{equation*}
\mbox{(VI)}\left\{
\begin{array}{ll}
u\in H_0^1(\Omega)&\\
\int_\Omega\nabla u(x)\cdot\nabla\left(v(x)-u(x)\right)\gl\int_\Omega
p\left(x,u(x)\right)\left(v(x)-u(x)\right),$${\forall}$$ v\in H_0^1(\Omega)&\\
u(x),v(x)\ls\psi(x),x\in \Omega,
\end{array}
\right.
\end{equation*}
where $\Omega$ is a bounded domain of $\mathbb{R}^N$, with smooth boundary $\partial\Omega, \psi:\Omega\to\mathbb{R}$ is the obstacle, $\psi\in H_0^1(\Omega)$ and $\psi|_{\partial\Omega}\gl 0$.

Variational inequalities (VIs) were initially introduced and studied by Stampacchia [7] in
1964. Since then, VIs have extensively been studied and applied to a large variety of applied
problems arising from structural analysis, economics, optimization, management science,
operations research and engineering sciences. Many results have been obtained for the existence of positive solutions or non-trivial solutions to VIs in the past three decades. For example, for the above variational inequality, M.Girardi, L.Mastroeni and M.Matzeu gave the existence result of its positive solutions. They proved  in [6] the existence of non-negative solutions for  variational inequality (VI) through the use of some estimates for the Mountain-Pass critical points obtained for the penalized equations associated with the (VI), and then proved the positivity of solution by a regularity result and the strong maximum principle. For more results on the existence of positive solutions or non-trivial solutions to VIs, one can refer to literature [1-6].

The main purpose of the present paper is to give some results for the existence of sign-changing solutions of the (VI). In the last two decades, people have extensively studied the existence of sign-changing solutions for various kinds of differential boundary value problems, see [9-13]. However, to the best of our knowledge, no one has studied the sign-changing solutions of (VIs). We need to point out that the study of the sign-changing solutions of VIs has extensive and important practical significance. Now, let's look at a practical example of studying the sign-changing solution of a variational inequality. Let us suppose that we have a beam that is fixed at one end, and there is a support underneath it, which exerts a force on the other end. When the force is small, the beam simply bends, and the variational inequality describing the beam bending  would have a positive solution. On the other hand, when the force is very large, the beam will have a complex bending shape, and the variational inequality describing the beam bending would have a  sign-changing solution.

 To show the existence of sign-changing solution to the variational inequality (VI),  we will use the penalty method, and the method of invariant set of descending flow proposed by Sun Jingxian; see [8]. In the last two decades, people have extensively studied the existence of sign-changing solutions for various kinds of differential boundary value problems by using the  method of invariant set of descending flow. This method has a simple and intuitive geometric meaning, and is especially suitable for proving the existence of the sign-changing solution.

Now let us introduce the main results of this paper. Throughout the paper we denote by $H_0^1(\Omega)$ the usual Sobolev space equipped with the inner product $\left\langle u,v\right\rangle=\int_{\Omega}\nabla u\cdot \nabla v\,dx$ and the norm$$\|u\|=\left(\int_\Omega|\nabla u|^2\,dx\right)^\frac{1}{2}$$
and by $L^q(\Omega)$, with $q\in [1,+\infty)$, the usual Lebesgue space with the norm defined as$$|u|_q=\left(\int_\Omega|u|^q\,dx\right)^\frac{1}{q}.$$
Then,  $H_0^1(\Omega)$ compact embedded $L^q(\Omega)$ for $q\in(2,2^*)$, where $2^*=\frac{2N}{N-2}$. We shall denote by $\mathcal{S}_q$ the imbedding constant with $|u|_q\ls \mathcal{S}_q\|u\|~~~~~~~~~~\text{for all}~~u\in H_0^1(\Omega)$.

Let $P_0=\left\{u\in H_0^1(\Omega):u(x)\gl 0~ \ \text{a.e. in}~\Omega\right\}$, then $P_0$ is a closed convex cone on $H_0^1(\Omega)$. We also denote by $X=C_0^1(\Omega)=\left\{u\in C^1(\Omega):u(x)=0,\forall x\in\partial\Omega\right\},$
and let $P=P_0\cap X=\left\{u\in X:u(x)\gl 0\  \text{on}~\bar\Omega\right\}$.

We denote by $v^+$ and $v^-$ respectively the positive and negative part of a function $v$, that is $v^+=\max\{v,0\}$ and $v^-=\max\{-v,0\}$.

Denote the first eigenvalue of Laplacian operator $-\Delta$ as $\lambda_1$, that is
$$\ \lambda_1=\inf\limits_{u\in H_0^1(\Omega)\backslash\{0\}}\frac{\int_\Omega |\nabla u|^2\,dx}{\int_\Omega |u|^2\,dx}.$$
Denote the second eigenvalue  of $-\Delta$ as $\lambda_2$, and let $\varphi_1,\varphi_2$ be the eigenfunctions corresponding to $\lambda_1,\lambda_2$ respectively. Let $Y=\text{span}\{\varphi_1,\varphi_2\}$.

For function $p:\bar\Omega\times \mathbb{R}\to \mathbb{R}$ we makes the following assumptions:

${\rm{(}}{{\rm{H}}_{\rm{1}}}{\rm{)}}$~ $p(x,\xi)$ is measurable in $x\in\bar{\Omega}$,  continuous in~$\xi\in \mathbb{R}$, and for some ~$a_1,a_2>0$, $$|p\left(x,\xi\right)|\ls a_1+a_2|\xi|^s,\ \forall(x,\xi)\in \bar\Omega\times \mathbb{R}$$
with $1<s<\frac{N+2}{N-2}=2^*-1$ if $N\gl 3$, and $s>1$ If $N=2$;

${\rm{(}}{{\rm{H}}_{\rm{2}}}{\rm{)}}$~ $p\left(x,\xi\right)=o(|\xi|)$ as $\xi \to 0$;

${\rm{(}}{{\rm{H}}_{\rm{3}}}{\rm{)}}$~ $\xi p(x,\xi)>0$ for any $\xi\neq 0$;

${\rm{(}}{{\rm{H}}_{\rm{4}}}{\rm{)}}$~ There exists $r>0$, so that for all $ |\xi|\gl r, x\in \bar\Omega$, $$0<(s+1)P(x,\xi)\ls\xi p(x,\xi),$$
where $P(x,\xi)=\int_0^{\xi} p(x,t)\,dt$ for $x\in \bar\Omega$;

Before giving the next hypothesis, we can first determine two constants $a_3,a_4>0$, from ${\rm{(}}{{\rm{H}}_{\rm{4}}}{\rm{)}}$, such that for all $(x,\xi)\in\bar{\Omega}\times \mathbb{R}$,
\begin{equation}\label{tao}
P\left(x,\xi\right)\gl a_3|\xi|^{s+1}-a_4.
\end{equation}

Let
$$Q(u)=\frac{1}{2}\|u\|^2-
a_3\int_\Omega|u|^{s+1}\,dx+a_4|\Omega| \ \mbox{for}\ u\in H_0^1(\Omega).$$
${\rm{(}}{{\rm{H}}_{\rm{5}}}{\rm{)}}$~  There exists $r_1>0$, such that $u\ls\psi$ for all $u\in  \bar B(0,r_1)\cap Y$, and $Q(u)<0$ for all $u\in \partial B(0,r_1)\cap Y$.

\begin{remark}\label{R13}
\rm According to ${\rm{(}}{{\rm{H}}_{\rm{1}}}{\rm{)}}$ and ${\rm{(}}{{\rm{H}}_{\rm{2}}}{\rm{)}}$, we have for $(x,u)\in \bar{\Omega}\times \mathbb{R}$,
\begin{equation}\label{11}
\forall \delta >0,\exists C_\delta >0, |p\left(x,\xi\right)|\ls \delta|\xi|+C_\delta|\xi|^s.
\end{equation}
By integrating both sides of the inequality \eqref{11} over $\xi$, we have for $(x,u)\in \bar{\Omega}\times \mathbb{R}$,
\begin{equation}\label{12}
|P\left(x,\xi\right)|\ls \delta|\xi|^2+C_\delta|\xi|^{s+1}.
\end{equation}
\end{remark}

In this paper, we have the following main result.

\begin{theorem}\label{d11}
Let ${\rm{(}}{{\rm{H}}_{\rm{1}}}{\rm{)}}$~$\sim~
$${\rm{(}}{{\rm{H}}_{\rm{5}}}{\rm{)}}$ hold, and $s<2$, then the variational inequality (VI) has at least one positive solution, one negative solution and one sign-changing solution.
\end{theorem}

\renewcommand{\theequation}{\thesection.\arabic{equation}}
\setcounter{equation}{0}


\section{Sign-changing solutions for penalty problems}   \setcounter{equation}{0}
\rm In this section, we construct a penalty problem $(E)_\va$ related to (VI), and use the method of invariant set of descending flow to study the existence of sign-changing solutions for the penalty problem $(E)_\va$. During the proof process, we also give the existence results for positive and negative solutions for the problem $(E)_\va$.

First, we introduce the penalty problem related to the variational inequality $(\mbox{VI})$, that is, for any $\varepsilon>0$,
\begin{equation*}
\mbox{(E)}_\va\left\{
\begin{array}{ll}
u\in H_0^1(\Omega)&\\
\int_\Omega\nabla u\cdot\nabla v+\frac{1}{\varepsilon}\int_\Omega\left(u-\psi\right)^+v=
\int_\Omega p\left(x,u(x)\right)v,$${\forall}$$ v\in H_0^1(\Omega),\\
\end{array}
\right.
\end{equation*}
where $\left(u-\psi\right)^+$ denotes the positive part of the function $u-\psi$.

Let $Z_\varepsilon=H_0^1(\Omega)$ with the new norm:$$\|u\|_\varepsilon=\left(\int_\Omega(|\nabla u|^2+\frac{1}{\varepsilon}u^2)\,dx\right)^\frac{1}{2}.$$
Evidently $H_0^1(\Om)$ and $Z_\varepsilon$ are homeomorphic, $\|u\|$ is equivalent to $\|u\|_\varepsilon$.
Related to the  penalty problem $(E)_\va$, we consider the following functional:
\begin{equation}\label{25}
I_\varepsilon(u)=\frac{1}{2}\|u\|_{\varepsilon}^2+
\frac{1}{2\varepsilon}\int_\Omega
\left(\left(u-\psi(x)\right)^+\right)^2\,dx-\int_\Omega P(x,u)-\frac{1}{2\varepsilon}|u|_2^2.
\end{equation}

Let us define the nonlinear operator $A_\va: X\ri X$ by $A_\varepsilon=K_\varepsilon F_\varepsilon$,  where $K_\varepsilon: C(\bar\Om)\ri X$ is the solution operator of the following  boundary value problem:
\begin{equation*}
\left\{
\begin{array}{ll}
-\Delta u+\frac{1}{\varepsilon}u=v,&x\in \Omega,\\
u=0,&x\in \partial\Omega,\\
\end{array}
\right.
\end{equation*}
and $F_\varepsilon: X\to C(\bar\Omega)$ is a nonlinear operator defined by $$F_\varepsilon u=p(x,u)+\frac{1}{\varepsilon}u-\frac{1}{\varepsilon}\left(u-\psi\right)^+,  \forall u\in X.$$ Obviously, the gradient operator of $I_\varepsilon(u)$ with respect to the norm $\|\cdot\|_\varepsilon$ is $I'_\va(u)=u-A_\varepsilon u$ for any $u\in X$.

Consider the initial value problem
\begin{equation}\label{23}
\left\{
\begin{array}{ll}
\displaystyle{\frac{\,du}{\,dt}}=A_{\varepsilon} u-u,&\\
u(0)=u_{0}\in X\backslash K,\\
\end{array}
\right.
\end{equation}
where $K=\{u\in X:u=A_\varepsilon u\}$. Since $A_\varepsilon$ is Lipschitz continuous on $X$, the solution of \eqref{23} uniquely exists. Let $\left[0,\eta(u_0)\right)$ be the maximal right existence interval of the  solution $u(t,u_0)$, here $\eta(u_0)\ls +\infty$.
\vskip 0.1in

It follows from [8] we have the following definition.

\begin{definition}
\rm The nonempty subset $M$ of $X$ is called a descending flow invariant set of \eqref{23} if $$\{u(t,u_0):t\in\left[0,\eta(u_0)\right)\}\subset M,\forall u_0\in M\backslash K.$$
\end{definition}

 We say that the functional $I_\varepsilon$ satisfies the $(PS)$ condition, if for any sequence $\{u_n\}\subset Z_\va$, $\{I_\varepsilon(u_n)\}$ is bounded, and $I_\varepsilon^{\prime}(u_n)\to 0$ as $n\to \infty$, then $\{u_n\}$ has at least one convergent subsequence in $Z_\va$.

By a usual way we can easily show the following Lemma 2.1.

\begin{lemma}\label{y21}
$I_\varepsilon$ satisfies the $(PS)$ condition on $Z_\varepsilon$ for each $\va>0$.
\end{lemma}

\begin{lemma}\label{y22}
Let $\Lambda(u_0)=\left\{u(t,u_0):t\in \left[0,\eta(u_0)\right)\right\}$ be the orbits of \eqref{23} emitting from the initial point $u_0\in X$. Assume that $$\tilde{c}:=\inf\limits_{t\in\left[0,\eta(u_0)\right)}
I_\varepsilon\left(u(t)\right)>-\infty,$$ then the following two conclusions hold:\\
(1)\ there exists $\bar{u}_0\in K$, such that $\lim \limits_{t\to\eta^-(u_0)}u(t,u_0)=\bar{u}_0$, and $I_\varepsilon(\bar{u}_0)\ls I_\varepsilon(u_0)$;\\
(2)\  there exists a sequence $\{t_n\}\subset\left[0,\eta(u_0)\right)$ with $t_n\to\eta^-(u_0)$ as $n\to\infty$, such that $u(t_n,u_0)\to\bar{u}_0$ in $X$.
\end{lemma}
\begin{proof}
According to the theories of ordinary equation in Banach spaces, we see $\Lambda(u_0)\subset X$. For brevity we denote $u(t,u_0)$ as $u(t)$ for $t\in\left[0,\eta(u_0)\right)$. Then we have for $t\in\left[0,\eta(u_0)\right)$,
\begin{equation}\label{24}
\frac{\,dI_\varepsilon(u(t))}{\,dt}=\lge I_\varepsilon^{\prime}(u(t)), u^{\prime}(t)\rge=
-\|I_\varepsilon^{\prime}(u(t))\|_\varepsilon^2\ls 0.
\end{equation}
This implies that $I_\varepsilon(u(t))$ is non-increasing along the obits. For any $0\ls t_1<t_2<\eta(u_0)$, by the  H\"{o}lder inequality we have
\begin{equation}\label{25}
\begin{aligned}
\|u(t_2)-u(t_1)\|_\varepsilon&\ls\int_{t_1}^{t_2}\|u^{\prime}(t)\|_\varepsilon\,dt=
\int_{t_1}^{t_2}\|I_\varepsilon^{\prime}(u(t))\|_\varepsilon\,dt\\
&\ls\left(\int_{t_1}^{t_2}\|I_\varepsilon^{\prime}(u(t))\|_\varepsilon^2\,dt\right)^
\frac{1}{2}(t_2-t_1)^\frac{1}{2}.
\end{aligned}
\end{equation}
It follows from \eqref{24} that
\begin{equation}\label{26}
\begin{aligned}
\int_{t_1}^{t_2}\|I_\varepsilon^{\prime}(u(t))\|_\varepsilon^2\,dt&=-\int_{t_1}^{t_2}
\frac{\,dI_\varepsilon(u(t))}{\,dt}\,dt\\
&=I_\varepsilon(u(t_1))-I_\varepsilon(u(t_2))\\
&\ls I_\varepsilon(u(t_0))-\tilde{c}<\infty.
\end{aligned}
\end{equation}
It follows from \eqref{25} and \eqref{26} that
\begin{equation}\label{27}
\|u(t_2)-u(t_1)\|_\varepsilon\ls\left(I_\varepsilon(u(t_0))-\tilde{c}\right)^\frac{1}{2}
(t_2-t_1)^\frac{1}{2}.
\end{equation}
If $\eta(u_0)<+\infty$, then
$$\|u(t_2)-u(t_1)\|_\varepsilon\to 0\ \mbox{as}\  t_1, t_2\to\eta^-(u_0),$$
and so $\lim\limits_{t\to\eta^-(u_0)}u(t)=
\bar{u}_0$ for some $\bar{u}_0\in K$. Obviously, for any $t_n\to\eta^-(u_0)$, we have $$\tilde{c}\ls I_\varepsilon(u(t_n))\ls I_\varepsilon(u_0)$$
and
\begin{equation}\label{28}
I_\varepsilon^{\prime}(u(t_n))\to 0~\text{in}~Z_\va~\text{as}~n\to \infty.
\end{equation}
If $\eta(u_0)=+\infty$, it follows from (2.5) that
$$\int_0^{+\infty} \|I_\varepsilon^{\prime}(u(t_n))\|_\varepsilon^2\,dt<+\infty.$$
So, there also exists a sequence $\{t_n\}\subset\left[0,+\infty\right)$ such that \eqref{28} holds. It follows from Lemma \ref{y21} that we may assume, up to a subsequence, that $u(t_n)\to\bar{u}_0$ in $Z_\varepsilon$ for some $\bar{u}_0\in K.$
By a usual way we can prove that $\lim\limits_{t\to +\infty}u(t)=\bar{u}_0$ in $Z_\varepsilon$.

Let $X_0=W^{2,r}(\Omega)$ for some $r\gl2$ such that the embedding from $X_0$ into $X$ is compact. Take a finite sequence of Banach spaces $\{X_i\}$ such that
$$X_0\hookrightarrow X_1\hookrightarrow X_2\hookrightarrow \cdots\hookrightarrow X_m=Z_{\varepsilon}$$
and $A_\varepsilon$ is continuous and bounded from $\{X_i\}$ into $X_{i-1}$ for $i=1,2,\cdots,m$. It is easy to see that  the orbit $\Lambda(u_0)$ is bounded in $Z_{\varepsilon}$. Assume that
$$\sup\limits_{t\in\left[0,\eta(u_0)\right)}\|u(t)\|_\varepsilon\ls R_m.$$
By a direct computation we have
\begin{equation}\label{212}
u(t)=e^{-t}u_0+\int_0^t e^{-t+s}A_\varepsilon(u(s))\,ds \ \mbox{for}\ t\in\left[0,\eta(u_0)\right).
\end{equation}
Thus we have
\begin{equation*}
\begin{aligned}
\Big\|e^{-t_n}\int_0^{t_n}e^s A_\varepsilon(u(s))ds\Big\|_{X_{m-1}}\ls e^{-t_n}\int_0^{t_n}e^s \|A_\varepsilon(u(s))\|_{X_{m-1}}\,ds\ls R_{m-1}
\end{aligned}
\end{equation*}
for some $R_{m-1}>0$. By a finite steps we obtain that
\begin{equation*}
\Big\|e^{-t_n}\int_0^{t_n}e^s A_\varepsilon(u(s))ds\Big\|_{X_0}\ls e^{-t_n}\int_0^{t_n}e^s \|A_\varepsilon(u(s))\|_{X_0}\,ds\ls R_0
\end{equation*}
for some $R_0>0$. This implies that
$$\|u(t_n)-e^{-t_n}u_0\|_{X_0}\ls R_0.$$
Therefore, up to a subsequence if necessary,  the sequence $\{u(t_n)-e^{-t_n}u_0\}$ is convergent in $X$, and so
$u(t_n)\to \bar{u}_0$ as $n\to \infty $ in $X$. Since $I_\varepsilon(u(t))$ is non-increasing in
$t\in\left[0,\eta(u_0)\right)$, we have
$$I_\varepsilon(\bar{u}_0)\ls I_\varepsilon(u_0).$$
The proof is complete.
\end{proof}

\begin{lemma}\label{y23}
Let $M$ be an invariant set of descending flow of \eqref{23}, and $D\subset M$ be a closed invariant set of descending flow of \eqref{23} in $M$. Assume that
$$\tilde{c}:=\inf\limits_{u\in D}
I_\varepsilon\left(u(t)\right)>-\infty.$$
Then there exists a $u_0\in D\cap K$ such that $I_\varepsilon\left(u_0\right)=\tilde{c}$.
\end{lemma}
\begin{proof}
Take $\{u_n\}\subset D$ such that $\lim\limits_{n\to\infty}I_\varepsilon\left(u_n\right)=\tilde{c}, n=1,2,\cdots.$ For each $n=1,2,\cdots,$ it follows from Lemma \ref{y22} that
there exists a $\bar{u}_n\in K$ and a sequence of $\{u_{n,k}\}_{k=1}^{\infty}\subset \Lambda(u_n)$ such that $u_{n,k}\to\bar{u}_n$ in $X$ as $k\to\infty$, and $I_\varepsilon(\bar{u}_n)\ls I_\varepsilon(u_n)$. Since $\Lambda(u_n)\subset D$ and $D$ is closed in $M$, we have $\bar{u}_n\in D$ for each $n=1,2,\cdots.$
By using the (PS) condition we can see that $\{\bar{u}_n\}$ is bounded in $Z_\va$. By a bootstrap procedure, we see that $\{\bar{u}_n\}$ is bounded in $X$. Since $\bar{u}_n=A_\varepsilon(\bar{u}_n)$ and $A_\varepsilon:X\to X$ is compact, we may assume, up to a subsequence, that $\bar{u}_n\to u_0$ in $X$. Obviously,
$u_0\in K\cap D$ and $I_\varepsilon(u_0)=\tilde{c}$. The proof is complete.
\end{proof}

\begin{lemma}\label{y24}
([13]) Assume $U$ is bounded connected open set of $\mathbb{R}^2$ and $(0,0)\in U$,
then there exists a connected component $\Gamma^{\prime}$ of the boundary of $U$ such that each one side ray $l$ though the origin satisfies $l\cap \Gamma^{\prime}\neq \emptyset$.
\end{lemma}

Similar to [8], we the following Definition 2.2 and Lemma 2.5.
\begin{definition}
\rm Let $M$ and $D$ be invariant sets of \eqref{23} with $D\subset M$. Denote $$C_M(D)=\left\{u_0:u_0\in D~\text{or}~u_0\in M\backslash D
\text{and there is}~t^{\prime}\in \left[0,\eta(u_0)\right),\text{such that} ~u(t^{\prime},u_0)\in D\right\}.$$
If $D=C_M(D)$, then $D$ is called a complete descending flow invariant set of \eqref{23}.
\end{definition}

\begin{lemma}\label{y25}
Let $M$  be an invariant set of \eqref{23}, $D$  an open subset of $M$. Then\\
(1)~$C_M(D)$ is an open subset of $M$;\\
(2) if $C_M(D)\neq M, \partial_M C_M(D)$ is a complete descending flow invariant set of \eqref{23};\\
(3) if $C_M(D)\neq M$ and $\inf\limits_{u\in \partial_M D}\Phi(u)>-\infty$, then $$\inf\limits_{u\in \partial_MC_M(D)}\Phi(u)\gl\inf\limits_{u\in \partial_M D}\Phi(u).$$
\end{lemma}
\begin{lemma}\label{y26}
 Assume that condition ${\rm{(}}{{\rm{H}}_{\rm{1}}}{\rm{)}}\sim{\rm{(}}{{\rm{H}}
_{\rm{5}}}{\rm{)}}$ hold, then the problem $(2.1)$ has at least one positive solution, one negative solution and one sign-changing solution.
\end{lemma}
\begin{proof}
At first, we show that $\mathring{P}$ and $-\mathring{P}$ are invariant sets of descending flow of \eqref{23}. We only show that $\mathring{P}$ is an invariant set of descending flow, the proof of $-\mathring{P}$ can be the same. Since for any $u\in P$, one has
$$F_\varepsilon u=\frac{1}{\varepsilon}u-\frac{1}{\varepsilon}(u-\psi)^+ +p(x,u)\in
P,$$
also $K_\varepsilon (P)\subset P$, and thus $A_\varepsilon (P)\subset P$. Therefore, we know that $P$ is an invariant set of descending flow of (2.2). For any $u_0\in\mathring{P}$, the unique solution $u(t,u_0)$ of \eqref{23} satisfying
$$u(t,u_0)=e^{-t}u_0+\left(1-e^{-t}\right)\lim\limits_{n\to \infty}\frac{1}{n}\sum^{n}_{k=1}A_{\varepsilon}\left(u\left(\ln
\left(1+\frac{k}{n}(e^t-1)
\right)
\right),u_0\right).$$
Since $P$ is a closed convex set,
$$\lim\limits_{n\to \infty}\frac{1}{n}\sum^{n}_{k=1}A_{\varepsilon}\left(u\left(\ln
\left(1+\frac{k}{n}
(e^t-1)\right)
\right),u_0\right)\in P.$$
Note that $u_0\in\mathring{P}$. So $u(t,u_0)\in \mathring{P}$ for all $t\in \left[0,\eta(u_0)\right)$. Hence, $\mathring{P}$ is an invariant set of descending flow of (2.2). By \eqref{12}, we have
\begin{equation*}
\begin{aligned}
I_\varepsilon(u)&=\frac{1}{2}\|u\|^2+\frac{1}{2\varepsilon}\int_\Omega
\left(\left(u-\psi(x)\right)^+\right)^2\,dx-\int_\Omega P(x,u)\,dx\\
&\gl \frac{1}{2}\|u\|^2-\int_\Omega P(x,u)\,dx\\
&\gl\frac{1}{2}\|u\|^2-\delta|u|_2^2-C_\delta|u|_{2^*}^{2^*}\\
&\gl\left(\frac{1}{2}-\mathcal{S}_2 \delta\right)\|u\|^2-C_\delta\mathcal{S}_{2^*}\|u\|^{2^*},
\end{aligned}
\end{equation*}
where $\mathcal{S}_2,\mathcal{S}_{2^*}$  are Sobolev embedding constants, so there is a small enough $\rho_0>0$, such that   $\rho_0<r_1$, and
$$\alpha_0= \left(\frac{1}{2}-\mathcal{S}_2 \delta\right)\rho_0^2-C_\delta\mathcal{S}_{2^*}\rho_0^{2^*}>0.$$
Therefore, $I_\varepsilon(u)\gl\alpha_0$ for all $u\in Z_\va$ with $\|u\|=\rho_0$. According to \eqref{tao} and ${\rm{(}}{{\rm{H}}_{\rm{5}}}{\rm{)}}$, for any $u\in \partial B(0,r_1)\cap Y$, we have
\begin{equation}\label{113}
\begin{aligned}
I_\varepsilon(u)&=\frac{1}{2}\|u\|^2+\frac{1}{2\varepsilon}\int_\Omega
\left(\left(u-\psi(x)\right)^+\right)^2\,dx-\int_\Omega P(x,u)\,dx\\
&=\frac{1}{2}\|u\|^2-\int_\Omega P(x,u)\,dx\\
&\ls\frac{1}{2}\|u\|^2-a_3\int_\Omega|u|^s\,dx+a_4|\Omega|\\
&<0.
\end{aligned}
\end{equation}

Let $i_\varepsilon$ from $H_0^1(\Om)$ onto $Z_\varepsilon$ be the identity mapping.  Obviously, $i_\varepsilon$ is a homeomorphic mapping between $H_0^1(\Om)$ and $Z_\varepsilon$. Let $O=i_\va\left(B(0,\rho_0)\right)$, where $B(0,\rho_0)=\{u\in H_0^1(\Om):\|u\|<\rho_0\}$. Then $O$ is an open set in $Z_\varepsilon$, and so an open set in $X$ since $X\hookrightarrow Z_\va$. Moreover, $I_\varepsilon(u)\gl\alpha_0$ for all $u\in \partial O$, where $\partial O$ denotes the boundary of $O$ in $Z_\varepsilon$. Let
$$G=\{u\in Y:\|u\|<r_1\},$$
obviously, $G\supset B(\theta,\rho_0)\cap Y$. Let $G_1=i_\varepsilon(G)$, then $G_1$ is a bounded open set in $Z_\varepsilon\cap Y$, and $G_1$ is an open set of $X\cap Y$  containing the origin. Take $\delta_1\in (0, \alpha_0)$ small enough, and let $D_\varepsilon$ be the connected component of $I_\varepsilon^{\delta_1}:=\{u\in X:I_\varepsilon(u)<\delta_1\}$ containing the origin point of $X$. Then $D_\varepsilon$ is an invariant set of descending flow of (2.2), and so $\mathring{P}\cap D_\varepsilon,-\mathring{P}\cap D_\varepsilon$ are two nonempty invariant sets of descending flow of (2.2). It follows from \eqref{113} that $C_P(\mathring{P}\cap D_\varepsilon)\neq P$. Thus, $\partial_P(\mathring{P}\cap D_\varepsilon)\neq \emptyset$ is a closed invariant set of descending flow of (2.2). By Lemma \ref{y25} we have
$$c_+=\inf\limits_{u\in\partial_PC_P\left(\mathring{P}\cap D_\varepsilon\right)}I_\varepsilon(u)\gl\inf\limits_{u\in\partial_P\left(
\mathring{P}\cap D_\varepsilon\right)}I_\varepsilon(u).$$
It follows from Lemma \ref{y23} that there exists $u_+^\varepsilon\in\partial_PC_P(\mathring{P}\cap D_\varepsilon)\cap K$ such that $I_\varepsilon(u_+^\va)=c_+$. Then $u_+^\varepsilon$ is a positive solution of $(E)_\varepsilon$.

Similarly, $(E)_\va$ has at least one negative solution $u_-^\varepsilon\in\partial_{-P}C_{-P}(-\mathring{P}\cap D_\varepsilon)$.

Next we shall prove the existence of the sign-changing solution of $(E)_\va$. It follows from \eqref{113} that $C_X(D_\varepsilon)\cap Y$ is a bounded open subset of $Y$ containing the origin of $Y$. According to Lemma \ref{y24}, there exists a connected component $\Sigma^{\prime}$ of $\partial C_X(D_\varepsilon)\cap Y$ such that each one sided ray $l$ through the origin of $Y$ satisfies $l\cap\Sigma^{\prime}\neq \emptyset$. Let $\Sigma$ be the connected component of $\partial_{X} C_X(D_\varepsilon)$ containing $\Sigma^{\prime}$. Obviously, $C_{\Sigma}(\mathring{P}\cap \Sigma),C_{\Sigma}(-\mathring{P}\cap \Sigma)$) are two nonempty open subsets of $\Sigma$. By the connectedness of $\Sigma$ we see that
$$\Lambda:=\Sigma\backslash\left(C_{\Sigma}(\mathring{P}\cap \Sigma)\cup C_{\Sigma}(-\mathring{P}\cap \Sigma)\right)\neq \emptyset.$$
Since $\Sigma$ is an invariant set of descending flow of \eqref{23}, $C_{\Sigma}(\mathring{P}\cap \Sigma)$ and $C_{\Sigma}(-\mathring{P}\cap \Sigma)$ are two complete invariant set of descending flow of \eqref{23} in $\Sigma$, $\Lambda$ is a closed invariant set of descending flow of \eqref{23} in $\Sigma$. Moreover,
$$\tilde{c}:=\inf\limits_{u\in \Lambda}
I_\varepsilon(u)\gl\inf\limits_{u\in \partial B(0,\rho_0)}
I_\varepsilon(u)>-\infty.$$
It follows from Lemma \ref{y23} that $\tilde{c}$ is a critical value of $I_\varepsilon$, and $I_\varepsilon(\omega_{\varepsilon})=\tilde{c}$ for some $\omega_{\varepsilon}\in \Lambda\cap K$. Obviously, $\omega_{\varepsilon}\notin P\cup(-P)$ and $\omega_{\varepsilon}\notin B(0,\rho_0)$. Thus, $\omega_{\varepsilon}$ is a sign-changing solution of $(E)_\va$. The proof is complete.
\end{proof}

\section{Sign-changing solutions of the variational inequality}
\setcounter{equation}{0}

In this section, in order to obtain the sign-changing solution of (VI), we take a sequence of positive numbers $\{\va_n\}$ such that $\va_n\ri 0$ as $n\ri\infty$. By the method of section 2 we can obtain  a sign-changing solution sequence $\{\omega_{\varepsilon_n}\}$ of the penalty problem $(E)_{\va_n}$. Then, in this section we can show that, up to a subsequence if necessary, $\om_{\va_n}\ri \om$ as $n\ri \infty$ for some $\om\in H_0^1(\Om)$, and $\om$ is a sign-changing solution of (VI). For convenience, in this section, for each $\va_n>0$, the  penalty equation $(E)_{\va_n}$ is denoted as $(E)_n$, the corresponding functional is denoted as $I_n$, and the sign-changing solution $\omega_{\varepsilon_n}$ is denoted as $\omega_n$.

\begin{lemma}\label{y131}
 There exists $\sigma>0$, such that for any $\varepsilon_n>0$, one has $0\ls I_n(\omega_n)\ls \sigma$.
\end{lemma}
\begin{proof}
According to the proof process in Lemma \ref{y26}, there must be a ray passing through the origin and passing through a point $u\in\Lambda$, and so,
$$I_n(\omega_n)\ls\max
\limits_{(t,u)\in[0,1]\times ( B(0,r_1)\cap Y)}I_n(tu).$$
Since $u\ls\psi$ for any $u\in  \bar B(0,r_1)\cap Y$, one has
$$\int_\Omega\left((
tu-\psi)^+\right)^2\,dx=0~~~~(t\in[0,1]).$$
Then,
$$I_n(\omega_n)\ls \max\limits_{(t,u)\in[0,1]\times (B(0,r_1)\cap Y)}\left(\frac{t^2}{2}\|u\|^2-\int_\Omega P(x,tu(x))\,dx\right).$$
Since $[0,1]\times (\bar B(0,r_1)\cap Y)$ is compact, by \eqref{12}, one has  $0\ls I_n(\omega_n)\ls \sigma$ for some $\si>0$.
The proof is complete.
\end{proof}

\begin{lemma}\label{y132}
\rm Let $\omega_{n}$ be the sign-changing solution of problem $(E)_{n}$ given in Lemma \ref{y26}. If  $s<2$, then the solution sequence $\{\omega_n\}$ is bounded in $H_0^1(\Omega)$, that is, there is a positive constant  $c$ ($c$ does not depend on $\varepsilon_n$), such that $\|\omega_n\|\ls c, n=1,2,\cdots$.
\end{lemma}
\begin{proof} By Lemma 3.1 we have
$$\frac{1}{2}\int_\Omega|\na \om_n|^2+\frac{1}{\va_n}\int_\Om \Big(\int_0^{\om_n(x)}(s-\psi(x))^+ds\Big)dx\ls \si+\int_\Om P(x, \om_n(x))dx.$$
So, by (H$_4)$,
$$\frac{1}{2}\int_\Omega|\na \om_n|^2+\frac{1}{\va_n}\int_\Om \Big(\int_0^{\om_n(x)}(s-\psi(x))^+ds\Big)dx\ls\mbox{const}+\frac{1}{s+1}\int_\Om p(x, \om_n(x))\om_n(x)dx.$$
Thus, as $\om_n$ is a solution of $(E)_{n}$ we have
\begin{equation*}
\begin{aligned}
\big(\frac{1}{2}-\frac{1}{s+1}\big)\int_\Omega|\na \om_n|^2
\ls \mbox{const}+\frac{1}{(s+1)\va_n}\int_\Om (\om_n(x)-\psi(x))^+\om_n(x)\\
-\frac{1}{\va_n}\int_\Om \Big(\int_0^{\om_n(x)}(s-\psi(x))^+ds\Big)dx.
\end{aligned}
\end{equation*}
Set $\Om_n=\{x\in \Om: \om_n(x)\gl \psi(x)\}$ we have
\begin{equation*}
\begin{aligned}
\big(\frac{1}{2}-\frac{1}{s+1}\big)\int_\Omega|\na \om_n|^2
\ls \mbox{const}+\frac{1}{\va_n}\Big\{\frac{1}{(s+1)}\int_{\Om_n} (\om_n(x)-\psi(x))\om_n(x)\\
-\frac{1}{2}\int_{\Om_n}(\om_n(x)-\psi(x))^2ds\Big\}.
\end{aligned}
\end{equation*}
Since $s>1$,
\begin{equation}
\begin{aligned}
\int_\Omega|\na \om_n|^2
\ls \mbox{const}+\frac{1}{2\va_n}\int_{\Om_n} (\om_n(x)-\psi(x))\psi(x).
\end{aligned}
\end{equation}
Taking $v=\psi$ as a test function in $(E)_n$ we have
\begin{equation}
\begin{aligned}
\frac{1}{\va_n}\int_{\Om_n} (\om_n(x)-\psi(x))\psi(x)=-\frac{1}{2}\int_{\Om_n}\na \om_n \na\psi+\int_{\Om_n}p(x, \om_n(x))\psi(x).
\end{aligned}
\end{equation}
By (H$_2)$, (3.1) and (3.2) and the fact that $\psi\in L^q(\Om)$ for all $q\in (2, 2^*)$, we have
\begin{equation*}
\begin{aligned}
\int_\Omega|\na \om_n|^2\ls \mbox{const}\left(1+\left(\int_\Om|\om_n|^{2^*}\right)^{\frac{s}{2^*}}\right)
\end{aligned}
\end{equation*}
and by the continuous embedding of $H_0^1(\Om)$ into $L^{2^*}(\Om)$, $\|\om_n\|^2\ls \mbox{const}(1+\|\om_n\|^s)$, so $\{\om_n\}$ is bounded in $H_0^1(\Om)$ since $s<2$.
 The proof is complete.
\end{proof}

\begin{lemma}\label{y32}
There exists $\rho_1>0,~\alpha_+>0, \alpha_->0$, such that for any $\varepsilon_n>0,~\|\omega_{n}\|\gl \rho_1,~\|\omega_{n}^+\|\gl \alpha_+,~ \|\omega_{n}^-\|\gl\alpha_-$~.
\end{lemma}
\begin{proof}
Since $\omega_{n}$ is the sign-changing solution of $(E)_n$, taking $\omega_{n}$ as the test function, we get
\begin{equation*}
\begin{aligned}
\|\omega_{n}\|^2+\frac{1}{\varepsilon_n}\int_\Omega\left(
\omega_{n}-\psi\right)^+\omega_{n}&=\int_\Omega p(x,\omega_{n})\omega_{n}\\
&\ls \delta\int_{\Omega}|\omega_{n}|^2+C_\delta\int_{\Omega}
|\omega_{n}|^{s+1}\\
&\ls \delta\mathcal{S}_{2}\|\omega_{n}\|^2+C_\delta\mathcal{S}_{s+1}
\|\omega_{n}\|^{s+1}.
\end{aligned}
\end{equation*}
Note that $s+1>2$ in the above equation, then there exists $\rho_1>0$, so that for any $n=1,2,\cdots$,
$$\|\omega_{n}\|\gl \rho_1.$$
Take $\omega_{n}^+$ as the test function of $(E)_n$, then
$$\int_\Omega \nabla\omega_{n}^+\cdot \nabla\omega_{n}^+ +\frac{1}{\varepsilon_n}\int_\Omega\left(
\omega_{n}-\psi\right)^+\omega_{n}^+=\int_\Omega p(x,\omega_{n}^+)\omega_{n}^+.$$
therefore,
\begin{equation*}
\begin{aligned}
\|\omega_{n}^+\|^2+\frac{1}{\varepsilon_n}\int_\Omega\left(
\omega_{n}-\psi\right)^+\omega_{n}^+&\ls\delta
\int_\Omega
|\omega_{n}^+|\cdot\omega_{n}^+ +C_\delta\int_\Omega
|\omega_{n}^+|^{s}\cdot\omega_{n}^+\\
&\ls \delta\int_{\Omega}|\omega_{n}^+|^2+
C_\delta\int_{\Omega}|\omega_{n}^+|^{s+1}\\
&\ls \delta\mathcal{S}_{2}\|\omega_{n}^+\|^2+C_\delta\mathcal{S}_{s+1}
\|\omega_{n}^+\|^{s+1}.
\end{aligned}
\end{equation*}
Since $s+1>2$, there exists $ \alpha_+>0$, so that for any $\varepsilon_n>0$,
$$\|\omega_{n}^+\|\gl \alpha_+.$$
Similarly, we know that there exists  $\alpha_->0$, such that for any $\varepsilon_n>0$,$$\|\omega_{n}^-\|\gl \alpha_-.$$
The proof is complete.
\end{proof}

\begin{lemma}\label{y33}
Assume that condition ${\rm{(}}{{\rm{H}}_{\rm{1}}}{\rm{)}}$~$\sim~$${\rm{(}}{{\rm{H}}
_{\rm{5}}}{\rm{)}}$ hold and $s<2$, then (VI) has at least one sign-changing solution.
\end{lemma}
\begin{proof}
Let $\{\omega_{n}\}$ be the solution of $(E)_n$ given in Lemma \ref{y26}. According to Lemma \ref{y132}, $\{\omega_n\}$ is bounded in $H_0^1(\Omega)$, so we may assume that on $H_0^1(\Omega)$,
\begin{equation}\label{36}
\omega_n\rightharpoonup \omega, n\to \infty.
\end{equation}

First, prove $\omega$ is the solution of $(\mbox{VI})$. Combining Lemma \ref{y131}, \ref{y132} and \eqref{12}, we can get
\begin{equation*}
\begin{aligned}
\frac{1}{\varepsilon_n}
\|(\omega_n-\psi(x))^+\|_2^2&=\sigma-\frac{1}{2}\|\omega_n\|_2^2+\int_\Omega P(x,\omega_n)\,dx\\
&\ls\sigma+\delta\int_\Omega
|\omega_n|^2\,dx+C_\delta\int_\Omega
|\omega_n|^{s+1}\,dx\\
&\ls\sigma+ \delta\mathcal{S}_{2}\|\omega_n\|^2+C_\delta\mathcal{S}_{s+1}
\|\omega_n\|^{s+1}\\
&\ls \bar{c},
\end{aligned}
\end{equation*}
that is $|(\omega_n-\psi(x))^+|_2\ls \sqrt{\bar{c}\varepsilon_n}$. Hence, when
$\varepsilon_n\to 0$,
$$\left(
\omega_n-\psi(x)\right)^+\to 0.$$
By \eqref{36}, we can know on $L^2(\Omega)$,
$$\left(
\omega_n-\psi\right)^+\to\left(
\omega-\psi\right)^+\  \mbox{as}\ ~n\to \infty.$$
So that we deduce that $\left(
\omega-\psi\right)^+=0$ a.e. in $\Omega$, that is $\omega\ls \psi$ a.e. in $\Omega$.

Now we claim $\{\omega_n\}$ strongly converges to $\omega$ as $n\to \infty.$ Taking $\omega_n$ as a test function in $(E)_n$ we have
\begin{equation}\label{37}
\|\omega_n\|^2+\frac{1}{\varepsilon_n}
\int_\Omega\left(
\omega_n-\psi\right)^+\omega_n=\int_\Omega p(x,\omega_n)\omega_n.
\end{equation}
While, taking $\omega$ as a test function in $(E)_n$ we obtain
\begin{equation}\label{38}
\int_\Omega \nabla\omega_n\cdot \nabla\omega +\frac{1}{\varepsilon_n}\int_\Omega\left(
\omega_n-\psi\right)^+\omega=\int_\Omega p(x,\omega_n)\omega.
\end{equation}
Passing to the limit as $n\to \infty$ in the latter two relations and taking into account \eqref{36} and the dominated convergence theorem, we have
\begin{equation*}
\begin{aligned}
\varlimsup_{n\to \infty}\|\omega_n\|^2&\ls\int_\Omega p(x,\omega)\omega-\varliminf_{n\to \infty}\frac{1}{\varepsilon_n}\int_\Omega\left(
\omega_n-\psi\right)^+\omega_n\\
&=\|\omega\|^2+\lim_{n\to \infty}\frac{1}{\varepsilon_n}\int_\Omega\left(
\omega_n-\psi\right)^+\omega-\varliminf_{n\to \infty}\frac{1}{\varepsilon_n}\int_\Omega\left(
\omega_n-\psi\right)^+\omega_n\\
&=\|\omega\|^2+\varlimsup_{n\to \infty}\left[\frac{1}{\varepsilon_n}\int_\Omega\left(
\omega_n-\psi\right)^+(\omega-\psi)-\frac{1}{\varepsilon_n}
\int_\Omega
\left(\left(
\omega_n-\psi\right)^+\right)^2\right]\\
&\ls\|\omega\|^2\ls\varliminf_{n\to \infty}\|\omega_n\|^2,
\end{aligned}
\end{equation*}
here we also use the weak l.s.c of the norm and the fact that $\omega\ls\psi$ a.e. in $\Omega$. Hence, $$\|\omega_n\|\to\|\omega\|\ \mbox{as}\ ~n\to \infty.$$
Since $H_0^1(\Omega)$ is a separable Hilbert space, the claim is proved.

Now we are ready to show that $\omega$ is a solution of (VI). Since $\omega_n$ is the sign-changing solution of $(E)_n$, taking $v-\omega_n$, with $v\in H_0^1(\Omega)$ and $v\ls\psi$ a.e.in $\Omega$, as a  test function, we have
\begin{equation*}
\int_\Omega \nabla\omega_n \nabla(v-\omega_n) +\frac{1}{\varepsilon_n}\int_\Omega\left(
\omega_n-\psi\right)^+(v-\omega_n)=\int_\Omega p(x,\omega_n)(v-\omega_n).
\end{equation*}
Thanks to the choice of $v$, one has
$$\frac{1}{\varepsilon_n}\int_\Omega\left(
\omega_n-\psi\right)^+(v-\omega_{n})\ls \frac{1}{\varepsilon_n}\int_\Omega\left(
\omega_{n}-\psi\right)^+(\psi-\omega_{n})\ls 0,$$
Therefore, \begin{equation}\label{39}
\int_\Omega \nabla\omega_n \nabla(v-\omega_n)\gl\int_\Omega p(x,\omega_n)(v-\omega_n).
\end{equation}
Passing to the limit in \eqref{39}, we obtain
$$\int_\Omega \nabla\omega \nabla(v-\omega) \gl\int_\Omega p(x,\omega)(v-\omega).$$
Hence, $\omega$ is a solution of  (VI).

Since
$$\int_\Omega \nabla\omega_n\cdot \nabla\omega_n=\int_\Omega \nabla\omega_n^+\cdot \nabla\omega_n^+ +\int_\Omega \nabla\omega_n^-\cdot \nabla\omega_n^-,$$  that is $$\|\omega_n \|^2=\|\omega_n^+\|^2+\|\omega_n^-\|^2,$$
$\{\omega_n^+\},\{\omega_n^-\}$ are bounded.

Next, we shall prove $\omega_n^+\rightharpoonup \omega^+,\omega_n^-\rightharpoonup \omega^-.$ In fact, since  $\omega_n\to \omega$ in $H_0^1(\Omega)$, there is also $\omega_n\to \omega$ in $L^2(\Omega)$. For any $\varphi\in C_0^{\infty}(\Omega)$, using Green's formula and the dominated convergence theorem, one has
\begin{equation*}
\begin{aligned}
\left\langle \omega_n,\varphi\right\rangle&=\int_{\Omega}\nabla|\omega_n|\cdot\nabla\varphi
=-\int_{\Omega}|\omega_n|\Delta\varphi\\
&\to-\int_{\Omega}|\omega|\Delta\varphi=\int_{\Omega}\nabla|\omega|\cdot\nabla\varphi\\
&=\left\langle |\omega|,\varphi\right\rangle.
\end{aligned}
\end{equation*}
Since $C_0^{\infty}(\Omega)$ is dense in $H_0^1(\Omega)$, as known from Riesz theorem, $|\omega_n|\rightharpoonup|\omega|$. Due to
$$\omega^+=\frac{1}{2}(|\omega|+\omega),
\omega^-=\frac{1}{2}(|\omega|-\omega),$$
and $\omega_n\rightharpoonup \omega$, we can get
$\omega_n^+\rightharpoonup \omega^+,\omega_n^-\rightharpoonup \omega^-.$

Taking $\omega_n^+$ as a test function in $(E)_n$ we have
\begin{equation}\label{310}
\int_\Omega \nabla\omega_n\cdot \nabla\omega_n^+ +\frac{1}{\varepsilon_n}\int_\Omega\left(
\omega_n-\psi\right)^+\omega_n^+=\int_\Omega p(x,\omega_n)\omega_n^+.
\end{equation}
Taking $\omega^+$  as a test function in $(E)_n$ we obtain
\begin{equation}\label{311}
\int_\Omega \nabla\omega_n^+\cdot \nabla \omega^+ +\frac{1}{\varepsilon_n}\int_\Omega\left(
\omega_n-\psi\right)^+\omega^+=\int_\Omega p(x,\omega_n)\omega^+.
\end{equation}
Passing to the limit as $n\to \infty$ in the latter two relations, we have
\begin{equation*}
\begin{aligned}
\varlimsup_{n\to \infty}\|\omega_n^+\|^2&\ls\int_\Omega p(x,\omega_n)\omega_n^+ -\varliminf_{n\to \infty}\frac{1}{\varepsilon_n}\int_\Omega\left(
\omega_n-\psi\right)^+\omega_n^+\\
&=\|\omega^+\|^2+\lim_{n\to \infty}\frac{1}{\varepsilon_n}\int_\Omega\left(
\omega_n-\psi\right)^+\omega_n^+ -\varliminf_{n\to \infty}\frac{1}{\varepsilon_n}\int_\Omega\left(
\omega_n-\psi\right)^+\omega_n^+\\
&=\|\omega^+\|^2+\varlimsup_{n\to \infty}\left[\frac{1}{\varepsilon_n}\int_\Omega\left(
\omega_\varepsilon-\psi\right)^+(\omega^+ -\psi)-\frac{1}{\varepsilon_n}\int_\Omega
\left(\left(
\omega_n-\psi\right)^+\right)^2\right]\\
&\ls\|\omega^+\|^2\ls\varliminf_{n\to \infty}\|\omega_n^+\|^2.
\end{aligned}
\end{equation*}
Therefore, $\|\omega_n^+\|\to \|\omega^+\|$. Since $H_0^1(\Omega)$ is a separable Hilbert space, we get
$$\omega_n^+\to \omega^+~~\text{as}~~n\to \infty.$$
Similarly, it can be proved that on $H_0^1(\Omega)$,
$$\omega_n^-\to \omega^-~~\text{as}~~n\to \infty.$$
Also $\|\omega^+\|\gl\alpha_+>0,~\|\omega^-\|\gl\alpha_->0$, so $\omega=\omega^+-\omega^-$ is a sign-changing solution of (VI). The proof is complete.
\end{proof}

$\bm{{Proof~~of~~Theorem~~1.1.}}$

\begin{proof}
From Lemma \ref{y33}, we can see that problem (VI) has at least one sign-changing solution. For the proof process of the existence of its positive solution; see [6]. The existence of negative solution can be directly proved using the Mountain Pass Lemma on closed convex sets. The proof is complete.
\end{proof}

\begin{remark}\label{R13}
\rm In [6], M.Grardi, L.Mastroeni and M.Matzeu gave the regularity result of positive solution for (VI). In [1], Michele Matzeu and Raffaella Servadeil also gave the regularity results of positive solution for variational inequality problems with gradient terms, while the regularity results for negative solution and sign-changing solution have not yet been obtained.
\end{remark}



\end{document}